\documentclass[11pt]{article}
\usepackage[margin=1in]{geometry}

\usepackage{amsmath,amsthm,amssymb,hyperref,bm,amsbsy}

\usepackage{xcolor}
\usepackage{graphicx}
\usepackage[normalem]{ulem}
\usepackage{caption}
\usepackage{subcaption}
\usepackage{url}
\usepackage{multirow}

\numberwithin{equation}{section}

\usepackage{algorithm}
\usepackage[noend]{algorithmic}

\newcommand{\algorithmicbreak}{\textbf{break}}
\newcommand{\BREAK}{\STATE \algorithmicbreak}

\hypersetup{
    colorlinks=true, 
    linkcolor=blue,
    citecolor=blue,
    filecolor=blue,
    urlcolor=blue,
}

\theoremstyle{plain}
\newtheorem{theorem}{Theorem}[section]
\newtheorem{proposition}[theorem]{Proposition}
\newtheorem{lemma}[theorem]{Lemma}

\theoremstyle{definition}
\newtheorem{assumption}{Assumption}

\theoremstyle{remark}
\newtheorem{remark}[theorem]{Remark}

\usepackage{bbm}   

\newcommand{\op}{\operatorname}

\newcommand{\R}{\mathbb{R}}

\newcommand{\bE}{\mathbb{E}}
\newcommand{\dd}{\mathrm{d}}

\title{Dual Approaches to Stochastic Control via SPDEs and the Pathwise Hopf Formula}
\author{Mathieu Lauri\`{e}re\thanks{Shanghai Center for Data Science; NYU-ECNU Institute of Mathematical Science, NYU Shanghai, Shanghai, People’s Republic of China. Email: mathieu.lauriere@nyu.edu}
\and Jiefei Yang\thanks{NYU-ECNU Institute of Mathematical Sciences, NYU Shanghai, Shanghai, People’s Republic of China. Email: jy5595@nyu.edu.}
}
\date{}

\begin{document}
\maketitle

\begin{abstract}
We develop dual approaches for continuous-time stochastic control problems, enabling the computation of robust dual bounds in high-dimensional state and control spaces. Building on the dual formulation proposed in [L. C. G. Rogers, \textit{SIAM Journal on Control and Optimization, 46 (2007), pp. 1116--1132}], we first formulate the inner optimization problem as a stochastic partial differential equation (SPDE); the expectation of its solution yields the dual bound. Curse-of-dimensionality-free methods are proposed based on the Pontryagin maximum principle and the generalized Hopf formula. In the process, we prove the generalized Hopf formula, first introduced as a conjecture in [Y. T. Chow, J. Darbon, S. Osher, and W. Yin, \textit{Journal of Computational Physics 387 (2019), pp. 376--409}], under mild conditions. Numerical experiments demonstrate that our dual approaches effectively complement primal methods, including the deep BSDE method for solving high-dimensional PDEs and the deep actor-critic method in reinforcement learning. 
\end{abstract}



\section{Introduction}

Continuous-time stochastic optimal control theory studies optimal decision making for dynamical systems whose evolution is affected by randomness. Such problems arise naturally in finance~\cite{pham2009continuous}, operations research~\cite{heyman2004stochastic}, engineering~\cite{athans1971role}, and the study of large interacting populations~\cite{carmona2018probabilistic}. Furthermore, stochastic optimal control has a fundamental connection with reinforcement learning (RL), as RL can be interpreted as a set of methods designed to solve stochastic control problems, typically formulated as Markov decision processes~\cite{recht2019tour}. Theoretically, the solution to a stochastic optimal control problem can be characterized through nonlinear partial differential equations (PDEs), most notably the Hamilton--Jacobi--Bellman (HJB) equation, or through forward--backward stochastic differential equations (FBSDEs). 

Recently, machine learning-based methods have emerged as promising tools to solve stochastic control problems, especially in high dimensions. These include the deep BSDE method~\cite{E2017deepbsde}, the deep backward dynamic programming scheme~\cite{hure2020deep}, and the actor--critic algorithm~\cite{zhou2024solving}. However, a significant challenge remains: quantifying the approximation and generalization errors of deep neural networks is often theoretically out of reach. Consequently, these primal methods typically provide \emph{only an upper bound on the optimal value}, leaving the gap to the true solution unknown. To address this issue, we develop a \emph{dual approach} to complement existing primal methods, enabling the computation of \emph{both lower and upper bounds} for the optimal value.

In the primal formulation, one seeks a control process $u = (u_t)_{t}$ that minimizes a cost functional to achieve the value: 
\begin{equation*}
    V = \inf_u \mathbb{E}\left[ \int_0^T r(X_t^u, u_t)\,\dd t + g(X_T^u) \right], 
\end{equation*}
where the state $X_t^u$ follows a stochastic differential equation (SDE) driven by $u$. Building on the framework developed by Rogers~\cite{rogers2007pathwise}, which itself follows earlier work by \cite{ha1992deterministic}, the dual formulation transforms this minimization into a maximization problem over a set of martingales. While Rogers originally investigated discrete-time settings, subsequent studies extended these results to continuous-time and path-dependent cases~\cite{diehl2017stochastic, henry2016dual}. Specifically, the optimal value $V$ admits the representation:
\begin{equation} \label{eq:dual-intro}
    V = \sup_{M} \mathbb{E}\left[ \inf_{u} \left\{ \int_0^T r(X_t^u, u_t)\,\dd t + g(X_T^u) - M_T^u \right\} \right],
\end{equation}
where $M^u$ is a zero-mean martingale that depends on the controlled state $X^u$ (and thus implicitly on $u$). 

This primal-dual structure naturally yields error bounds: any suboptimal control from a primal method generates an upper bound on $V$, while the dual formulation provides a lower bound. Within the expectation in \eqref{eq:dual-intro}, the inner minimization degenerates into a pathwise optimal control problem. Together, these bounds offer a practical and reliable way to assess the quality of approximate solutions in high-dimensional settings, provided one can solve the dual problem efficiently, which is the main goal of this work.

\vskip 2pt
\noindent\textbf{Related works.} The primal--dual approach has a rich history in the literature, originating from dual formulations for optimal stopping problems~\cite{rogers2002monte, haugh2004pricing}. While this framework has been extensively developed for optimal stopping~\cite{belomestny2009true, desai2012pathwise, belomestny2019optimal, yang2024deep, bayer2025primal, ye2025deepmartingale}, its application to stochastic optimal control remains less explored, particularly regarding numerical implementation. The primary difficulty lies in the complexity of the inner problem: in optimal stopping, once a martingale is fixed, the inner problem reduces to finding a pathwise maximum. In contrast, stochastic optimal control requires solving a deterministic control problem over a potentially high-dimensional action space for every sampled path. Without an analytical solution, performing a separate numerical optimization for each path is computationally prohibitive.

Recent efforts have sought to mitigate this complexity in specific settings. In discrete-time or Markov decision processes, \cite{desai2012bounds} identified tractable cases, while \cite{belomestny2024primal} recently proposed a regression-based approach. Iterative methods and adversarial formulations have also emerged, such as the dual value estimates in~\cite{chen2024information}, the min--max game approach in~\cite{chen2025adversarial}, and dual value iteration for infinite-horizon problems~\cite{belomestny2026uvip}. In the continuous-time domain, existing work has largely focused on specific financial applications, such as Credit Valuation Adjustment (CVA)~\cite{henry2016dual, henry2017deep}, where the inner problem admits an analytical solution via an ODE. However, a general, curse-of-dimensionality-free numerical framework for continuous-time dual stochastic optimal control problems, independent of analytical tractability, is still lacking.

\vskip 2pt
\noindent\textbf{Main contributions.} 
This paper focuses on the efficient computation of the inner pathwise control problem in general settings, particularly where analytical solutions are unavailable or the control space is too large for exhaustive search. Our first contribution is a characterization of the value function for this inner problem via a stochastic partial differential equation (SPDE) in Stratonovich form (see Theorem~\ref{thm:spde}). 
To solve this SPDE numerically, we utilize a suboptimal martingale obtained from a primal approach and apply a Wong--Zakai type approximation. This reduces the problem to solving a first-order Hamilton--Jacobi (HJ) equation. Building on this, we propose two curse-of-dimensionality-free dual algorithms.
The first is based on Pontryagin's maximum principle for deterministic optimal control, leveraging its fundamental connection to the HJ equation (see \S~\ref{sec:pontryagin-method}).
The second utilizes the generalized Hopf formula, originally introduced as a conjecture in~\cite{chow2019algorithm}  (see \S~\ref{sec:hopf-method}). 
Notably, we provide a rigorous proof of the generalized Hopf formula under the assumptions of the maximum principle (see Theorem~\ref{thm:generalized-hopf}), a result that, to our knowledge, has not previously been proved in the literature. This ensures a true dual lower bound even when the inner optimization is not solved to exact optimality. Numerical experiments demonstrate that these approaches efficiently compute robust bounds for high-dimensional stochastic control problems, effectively complementing modern primal methods.

\vskip 2pt
\noindent\textbf{Organization of the paper.}
Section~\ref{sec:problem} states the problem and its dual formulation. In Section~\ref{sec:spde-inner-prob}, we derive the SPDE satisfied by the value function of the inner pathwise problem. In Section~\ref{sec:dual-algorithm}, we present the algorithms and the proof for the generalized Hopf formula. Numerical experiments are included in Section~\ref{sec:numerical-examples}. Finally, we conclude the paper with possible future directions in Section~\ref{sec:conclusion}. 

\section{Primal and dual formulations for stochastic control problems} \label{sec:problem}
Let $T>0$ be a finite time horizon. Let $\R^d$ be the state space and $U \subset \R^k$ be a non-empty separable metric space representing the action space. We consider a probability space $(\Omega, \mathcal{F}, \mathbb{P})$ equipped with a standard $d$-dimensional Brownian motion $W$ and with a natural filtration $\mathbb{F} = (\mathcal{F}_t)_{0\le t\le T}$. The finite-horizon continuous-time stochastic control problem is  
\begin{equation} \label{eq:sc-prob}
    \inf_{u} J[u] = \inf_{u} \bE\left[ \int_0^T r(X_t^u, u_t)\,\dd t + g(X_T^u) \right]
\end{equation}
subject to the controlled state dynamic 
\begin{equation} \label{eq:ito-dynamic}
    dX_t^u = b(X_t^u, u_t) \,\dd t + \sigma(X_t^u)\,\dd W_t, \quad X_0^u = x_0,
\end{equation}
where $x_0 \in \R^d$, $W_t$ denotes a $d$-dimensional Brownian motion, $b$ and $\sigma$ are given coefficient functions, and the control $u$ is taken from the admissible control set $\mathcal{U} := \{u:[0,T]\times \Omega \to U \subset \R^{k} \mid u \text{ is } \mathbb{F}\text{-adapted}\}$. 
The function $r(x, u): \R^d \times U \to \R$ and $g(x):\R^d \to \R$ can be regarded as the running and terminal cost, respectively. We define the value function of the stochastic optimal control problem by 
\begin{equation*}
    V(t, x) = \inf_u \bE\left[ \int_t^T r(X_s^{t, x, u}, u_s)\,\dd s + g(X_T^{t,x,u}) \right],
\end{equation*}
where $(X_s^{t, x, u})_{t\le s\le T}$ represents the state process starting from $X_t^{t,x,u} = x$. By definition, we have $V(0, x_0) = \inf_u J[u]$. 

Throughout the paper, we regard Problem~\eqref{eq:sc-prob} as the primal problem. Any suboptimal control produced by a numerical method yields an upper bound for $V(0,x_0)$. To complement this, we use the dual problem to compute a corresponding lower bound. Intuitively, instead of taking an infimum over all admissible controls, the dual formulation characterizes the optimal value through a supremum over a class of martingales. A well-approximated function in this dual class is therefore expected to produce a tight lower bound. 

This dual formulation was initially proposed in~\cite{rogers2007pathwise} for discrete-time controlled Markov processes and later extended to continuous time: Diehl et al. \cite{diehl2017stochastic} developed an extension using rough path analysis, while Henry-Labord\`{e}re et al. \cite{henry2016dual} proposed an alternative continuous-time approach based on a different methodology. Under the assumptions stated in Assumption \ref{assu:dual-sc-prob}, we summarize the dual formulation in Proposition \ref{prop:dual-formulation-sc}. 

\begin{assumption} \label{assu:dual-sc-prob}
    Assume that 
    \begin{enumerate}
        \item $b, r$ are uniformly bounded and continuous in $u$. 
        \item $b, \sigma, r$ are uniformly Lipschitz in $x$. 
        \item $g$ is Lipschitz continuous. 
    \end{enumerate}
\end{assumption}

\begin{proposition} \label{prop:dual-formulation-sc}
Under Assumption \ref{assu:dual-sc-prob}, the value function has a dual representation 
\begin{equation} \label{eq:dual-thm}
    V(t, x) = \sup_{h\in C_b^{1,2}}\bE\left[ \inf_{u \in \mathbb{U}} \left\{ \int_t^T r(X_s^{t, x, u}, u_s)\,\dd s + g(X_T^{t,x,u}) - \int_t^T \nabla_x h(s, X_s^{t,x,u})^\top \sigma(X_s^{t,x,u})\,\dd W_s \right\} \right],
\end{equation}
where $\mathbb{U}:= \{u:[0,T]\to U\mid u(\cdot)\text{ is measurable}\}$ and the supremum is achieved at $h^* = V$ provided $V\in C_b^{1,2}$. 
\end{proposition}

\begin{remark}
Compare~\eqref{eq:dual-thm} with~\eqref{eq:dual-intro}. We can see that 
\begin{equation*}
    M_T^u = \int_0^T \nabla_x h(t, X_t^u)^\top \sigma(X_t^u)\,\dd W_t. 
\end{equation*}
By the martingale representation theorem, a zero-mean martingale $M$ can be represented by an It\^{o} integral. Then optimizing over all zero-mean martingales is equivalent to optimizing over all functions $h$ inside the It\^{o} integral. 
\end{remark}

A large number of numerical methods focus on solving the primal problem directly via direct discretization and optimization, dynamic programming, the HJB equation, or FBSDE~\cite{han2016deep, E2017deepbsde, bachouch2022deep, zhou2024solving, li2024neural}, see a comprehensive overview on recent developments in~\cite{hu2024recent}. The HJB equation for the value function $V(t,x)$ is 
\begin{equation*}
\left\{\begin{aligned}
        \partial_t V(t,x) + \inf_{u\in U} G\left(t,x,u, \nabla_x V(t,x), \op{Hess}_x V(t,x)\right) &= 0, \\
        V(T,x) &= g(x), 
    \end{aligned}\right.
\end{equation*}
where $G(t,x,u,p,P) := \frac{1}{2}\op{Tr}(\sigma(x)\sigma(x)^\top P) + \langle p, b(x,u) \rangle + r(x,u)$, see~\cite{yong1999stochastic}. When only the drift coefficient is controlled, the HJB equation becomes a semilinear PDE: 
\begin{equation} \label{eq:nonlinear-pde}
    \partial_t V(t,x) + \frac{1}{2}\op{Tr}\left(\sigma(x)\sigma(x)^\top \op{Hess}_x V(t,x)\right) + f(x, \nabla_x V(t,x)) = 0, 
\end{equation}
where the nonlinearity $f(x, p) := \inf_{u\in U}\left\{ \langle p, b(x,u)\rangle + r(x,u) \right\}$. Thus, by the nonlinear Feynman-Kac formula, the value function satisfies the following FBSDE system 
\begin{equation*}
\left\{\begin{aligned}
    \tilde{X}_t &= \sigma(\tilde{X}_t) \,\dd W_t, \quad \tilde{X}_0 = x_0, \\
    V(t, \tilde{X}_t) &= g(\tilde{X}_T) + \int_t^T f(\tilde{X}_s, \nabla_x V(s, \tilde{X}_s))\,\dd s - \int_t^T Z_s\,\dd W_s, 
\end{aligned}\right. 
\end{equation*}
where $Z_s = \nabla_xV(s, \tilde{X}_s)^\top \sigma(\tilde{X}_s)$. 
Primal methods typically construct an approximate value function $\hat{V}(t,x)$ and/or an approximate optimal control $\hat{u}^*_t$ using neural networks or linear combinations of basis functions. 

We now describe how these primal approaches can be complemented by a dual method based on the dual formulation presented in Proposition \ref{prop:dual-formulation-sc}.  Our objective is to compute robust bounds for the primal problem~\eqref{eq:sc-prob} of the form  
\begin{equation*}
    V_{\text{lower}} \le V(0, x_0) \le V_{\text{upper}}:= J[\hat{u}^*].
\end{equation*}
Let $z_\theta(t,x) \approx \nabla_x V(t,x)^\top \sigma(x)$ be an approximation of the $Z$-component of the BSDE solution, where $\theta$ represents the parameters of the chosen approximation class\footnote{The function $z_\theta(t,x)$ serves as an approximation of the $Z$-component of the BSDE associated with the primal problem. This quantity can be directly obtained using the deep BSDE method~\cite{E2017deepbsde} or deep actor-critic method~\cite{zhou2024solving}, in which case $z_\theta$ is a neural network with parameters $\theta$. In settings where the Z-component is not explicitly available, it can instead be recovered from the value function approximation by setting $z_\theta(t, x) = \nabla_x \hat{V}_\theta(t,x)^\top \sigma(x)$,  where $\hat{V}_\theta(t,x)$ denotes a parameterized approximation of the value function.}. The approximation $z_\theta$ is useful to construct an almost optimal martingale by taking 
\begin{equation*}
    \hat{M}_T^u = \int_0^T z_\theta(t, X_t^u)\,\dd W_t. 
\end{equation*}
Then using Proposition \ref{prop:dual-formulation-sc} we obtain a lower bound 
\begin{equation} \label{eq:lower-bound}
    V_{\text{lower}} = \bE\left[ v_\omega(0,x_0) \right], 
\end{equation}
where for each $\omega \in \Omega$, 
\begin{equation} \label{eq:inner-optim-prob}
    v_\omega(t,x) = \inf_{u \in \mathbb{U}} \left\{ \int_t^T r(X_s^{t, x, u}(\omega), u_s)\,\dd s + g(X_T^{t,x,u}(\omega)) - \int_t^T z_\theta(s, X_s^{t,x,u}(\omega))\,\dd W_s(\omega) \right\}. 
\end{equation}
We refer to~\eqref{eq:inner-optim-prob} as the inner optimization problem. 

\section{SPDE for the inner optimization problem} \label{sec:spde-inner-prob}
In this section, we study the inner optimization problem~\eqref{eq:inner-optim-prob} via SPDE. We first present the dynamic programming equation in Lemma \ref{lem:dp-inner-optim} and then prove that $v_\omega$ satisfies a stochastic state-dependent Hamilton-Jacobi equation in Theorem \ref{thm:spde}. 

\begin{lemma} \label{lem:dp-inner-optim}
For any $\omega\in \Omega$, $(t,x)\in [0,T)\times \R^d$, $\hat{t}\in [t, T]$, 
\begin{equation} \label{eq:DP-pathwise}
    v_\omega(t,x) = \inf_{u\in \mathbb{U}} \left\{ \int_t^{\hat{t}}r(X_s^{t,x,u}(\omega), u_s)\,\dd s - \int_t^{\hat{t}} z_\theta(s,X_s^{t,x,u}(\omega))\,\dd W_s(\omega) + v_\omega(\hat{t}, X_{\hat{t}}^{t,x,u}(\omega)) \right\}.
\end{equation}
\end{lemma}

\begin{theorem} \label{thm:spde}
For any $\omega \in \Omega$, assume $v_\omega, z_\theta \in C^1([0,T]\times \R^d)$ and $\sigma \in C^1(\R^d)$. Then $v_\omega(t,x)$ satisfies the SPDE 
\begin{equation} \label{eq:spde-form}
\begin{aligned}
    &\quad \dd v + \inf_{u\in U} \left\{ \langle \nabla_x v, b(x,u) \rangle  + r(x,u)  \right\}\,\dd t + \left(\nabla_x v^\top \sigma(x) - z_\theta(t,x)\right)\circ \dd W_t(\omega) \\
    &+ \left\{\frac{1}{2}\op{Tr}\left( \sigma(x)J_xz_\theta(t,x)\right) - \langle \nabla_xv, \frac{1}{2}(\nabla \sigma:\sigma)(x) \rangle \right\}\,\dd t = 0, 
\end{aligned}
\end{equation}
with terminal condition $v(T,x) = g(x)$, where $J_x z$ is the Jacobian matrix of $z$ with respect to $x$, and $(\nabla\sigma:\sigma)_i = \sum_{j=1}^{d'}\sum_{k=1}^d (\partial_{x_k}\sigma_{ij})\sigma_{kj}$. 
\end{theorem}
\begin{proof}
By the state dynamic~\eqref{eq:ito-dynamic}, the $i$-th component of $X_t^u$, $i=1,\dots,d$, satisfies the Stratonovich SDE 
\begin{equation*}
    \dd X_t^{(i)} = \left( b^{(i)}(X_t, u_t) - \frac{1}{2} \sum_{j=1}^{d'} \langle \partial_{x} \sigma_{ij}(X_t), \sigma_{\cdot j}(X_t)\rangle \right)\,\dd t + \sum_{j=1}^{d'} \sigma_{ij}(X_t)\circ \dd W_t^{(j)}. 
\end{equation*}
Then apply the chain rule of of Stratonovich calculus to $v_\omega(t,X_t^u)$, we have 
\begin{equation} \label{eq:stratonovich-chain-rule}
\begin{aligned}
    &\quad \dd v_\omega = \partial_t v_\omega \,\dd t + \sum_{i=1}^d \partial_{x_i}v_\omega \circ \dd X_t^{u,(i)} \\
    &= \left(\partial_t v_\omega + \sum_{i=1}^d \partial_{x_i} v_\omega \left(b^{(i)}(X_t^u, u_t) - \frac{1}{2} \sum_{j=1}^{d'} \langle \partial_{x} \sigma_{ij}(X_t^u), \sigma_{\cdot j}(X_t^u)\rangle \right) \right)\,\dd t + \nabla_x v_\omega^\top \sigma(X_t^u) \circ \dd W_t. 
\end{aligned}
\end{equation}
In the following, we write $X_s^u := X_s^{t,x,u}(\omega)$ and $W_s := W_s(\omega)$ for simplicity and we will derive~\eqref{eq:spde-form} from Lemma \ref{lem:dp-inner-optim}. 
First, using Lemma \ref{lem:dp-inner-optim} and~\eqref{eq:stratonovich-chain-rule}, for any $u \in \mathbb{U}$, 
\begin{equation*}
\begin{aligned}
    0 &\le \int_t^{\hat{t}}r(X_s^{u}, u_s)\,\dd s - \int_t^{\hat{t}} z_\theta(s,X_s^{u})\,\dd W_s + v_\omega(\hat{t}, X_{\hat{t}}^{u}) - v_\omega(t,x) \\ 
    &= \int_t^{\hat{t}}\partial_s v_\omega \,\dd s + \int_t^{\hat{t}} \left(\sum_{i=1}^d \partial_{x_i} v_\omega(s, X_s^u) \left(b^{(i)}(X_s^u, u_s) - \frac{1}{2} \sum_{j=1}^{d'} \langle \partial_{x} \sigma_{ij}(X_s^u), \sigma_{\cdot j}(X_s^u)\rangle\right) + r(X_s^{u}, u_s) \right)\,\dd s \\
    &\quad + \int_t^{\hat{t}} \nabla_x v_\omega(s, X_s^u)^\top \sigma(X_s^u) \circ \dd W_t - \int_t^{\hat{t}} z_\theta(s, X_s^{u})\,\dd W_s
\end{aligned}
\end{equation*}
Let $\hat{t}\searrow t$. Then for all $u \in U$, 
\begin{equation*}
\begin{aligned}
    \dd v_\omega &+ \left\{ \sum_{i=1}^d \partial_{x_i} v_\omega(t,x)\left( b^{(i)}(x,u) - \frac{1}{2} \sum_{j=1}^{d'} \langle \partial_{x} \sigma_{ij}(x), \sigma_{\cdot j}(x)\rangle \right) + r(x,u) \right\}\,\dd t \\
    &+ \nabla_x v_\omega(t,x)^\top \sigma(x) \circ \dd W_t - z_\theta(t,x)\,\dd W_t \ge 0. 
\end{aligned}
\end{equation*}
Taking infimum leads to
\begin{equation} \label{eq:spde-1-side}
\begin{aligned}
    \dd v_\omega &+ \inf_{u\in U} \left\{ \sum_{i=1}^d \partial_{x_i} v_\omega(t,x)\left( b^{(i)}(x,u) - \frac{1}{2} \sum_{j=1}^{d'} \langle \partial_{x} \sigma_{ij}(x), \sigma_{\cdot j}(x)\rangle \right) + r(x,u) \right\}\,\dd t \\
    &+ \nabla_x v_\omega(t,x)^\top \sigma(x) \circ \dd W_t - z_\theta(t,x)\,\dd W_t \ge 0. 
\end{aligned}
\end{equation}

On the other hand, by~\eqref{eq:DP-pathwise}, for any $\varepsilon>0$, $\hat{t} - t>0$ small enough, there exists $u \in \mathbb{U}$ such that 
\begin{equation*}
    v_\omega(t,x) + \varepsilon(\hat{t} - t) \ge \int_t^{\hat{t}} r(X_s^{u}, u_s)\,\dd s - \int_t^{\hat{t}} z_\theta(s, X_s^{u})\,\dd W_s + v_\omega(\hat{t}, X_{\hat{t}}^{u}). 
\end{equation*}
It follows from~\eqref{eq:stratonovich-chain-rule} that 
\begin{equation*}
\begin{aligned}
    0 &\ge \int_t^{\hat{t}}\partial_s v_\omega\,\dd s + \int_t^{\hat{t}} \left(\sum_{i=1}^d \partial_{x_i} v_\omega(s, X_s^u) \left(b^{(i)}(X_s^u, u_s) - \frac{1}{2} \sum_{j=1}^{d'} \langle \partial_{x} \sigma_{ij}(X_s^u), \sigma_{\cdot j}(X_s^u)\rangle\right) + r(X_s^{u}, u_s) -\varepsilon \right)\,\dd s \\
    &\quad + \int_t^{\hat{t}} \nabla_x v_\omega(s, X_s^u)^\top \sigma(X_s^u) \circ \dd W_t - \int_t^{\hat{t}} z_\theta(s, X_s^{u})\,\dd W_s. 
\end{aligned}
\end{equation*}
Let $\hat{t}\searrow t$. Then for any $\epsilon>0$, there exists $u\in U$ such that 
\begin{equation*}
\begin{aligned}
    \dd v_\omega &+ \left\{ \sum_{i=1}^d \partial_{x_i} v_\omega(t,x)\left( b^{(i)}(x,u) - \frac{1}{2} \sum_{j=1}^{d'} \langle \partial_{x} \sigma_{ij}(x), \sigma_{\cdot j}(x)\rangle \right) + r(x,u) - \varepsilon \right\}\,\dd t \\
    &+ \nabla_x v_\omega(t,x)^\top \sigma(x) \circ \dd W_t - z_\theta(t,x)\,\dd W_t \le 0. 
\end{aligned}
\end{equation*}
Taking $\epsilon \to 0$ and together with~\eqref{eq:spde-1-side}, we obtain 
\begin{equation} \label{eq:mid-spde}
\begin{aligned}
    \dd v_\omega &+ \inf_{u\in U} \left\{ \sum_{i=1}^d \partial_{x_i} v_\omega(t,x)\left( b^{(i)}(x,u) - \frac{1}{2} \sum_{j=1}^{d'} \langle \partial_{x} \sigma_{ij}(x),  \sigma_{\cdot j}(x)\rangle \right) + r(x,u) \right\}\,\dd t \\
    &+ \nabla_x v_\omega(t,x)^\top \sigma(x) \circ \dd W_t - z_\theta(t,x)\,\dd W_t = 0. 
\end{aligned}
\end{equation}
Using the formula of Stratonovich integral, we obtain 
\begin{equation} \label{eq:zdw}
\begin{aligned}
    -z_\theta(t,x)\,\dd W_t &= \frac{1}{2}\sum_{j=1}^{d'} \sum_{k=1}^d \partial_{x_k} z_\theta^{(j)}(t,x) \sigma_{kj}(x) \,\dd t - \sum_{j=1}^{d'} z_\theta^{(j)}(t,x)\circ \dd W_t^{(j)}, \\
    &= \frac{1}{2}\op{Tr}\left(\sigma(x) J_xz_\theta(t,x)\right)\,\dd t - z_\theta(t,x)\circ \dd W_t. 
\end{aligned}
\end{equation}
Combining~\eqref{eq:mid-spde} and~\eqref{eq:zdw} lead to~\eqref{eq:spde-form}. 
\end{proof}

\begin{remark}
Nonlinear SPDEs have been studied in~\cite{lions1998fully} associated with the so-called pathwise stochastic control problem and later investigated in~\cite{buckdahn2007pathwise}. Given two independent Brownian motions $W, B$ and functions $g, r, b, \sigma, \theta$, the problem reads 
\begin{equation*}
    \inf_{\alpha} \bE\left[ g(X_T^\alpha) + \int_0^T r(X_t^\alpha, \alpha)\,\dd t \mid \mathcal{F}_T^B \right]
\end{equation*}
subject to 
\begin{equation*}
    \dd X_t^\alpha = b(X_t^\alpha, \alpha_t)\,\dd t + \sigma(X_t^\alpha, \alpha_t)\,\dd W_t + \theta(t, X_t^\alpha)\circ \dd B_t.  
\end{equation*}
Comparing this with the inner optimization problem~\eqref{eq:inner-optim-prob}, the objective functional in~\eqref{eq:inner-optim-prob} includes an additional It\^{o} integral $\int_t^T z_\theta(s, X_s^u)\,\dd W_s$. In theory, this issue could be handled by adding extra state variables $M_t^{(j)} = \int_0^t z_{\theta,j}(s, X_s^u)\,\dd W_s^{(j)}$ for $j=1,\dots, d$ and then applying the result in~\cite{buckdahn2007pathwise}. However, this state augmentation introduces $d$ additional state dimensions into the corresponding SPDE, substantially increasing the computational burden. For this reason, in Theorem \ref{thm:spde} we restrict our analysis to the special structure arising from the dual formulation of the classical stochastic control problem, which avoids this dimensionality issue.
\end{remark}

\section{Dual algorithm via computing SPDE} \label{sec:dual-algorithm}
In this section, we present the dual algorithm as a complement to primal methods using the neural network approximation. In principle, other types of primal method can also be incorporated into our dual approach as long as $\nabla_x V(t,x)$ can be approximated. We focus on neural networks due to their strong empirical performance in high-dimensional applications. As suggested in~\cite{rogers2007pathwise}, the dual approach may be particularly advantageous in high dimensions, where assessing the accuracy of value function approximations becomes challenging. Moreover, for stochastic control problems with high-dimensional state and/or control spaces, the curse of dimensionality renders many classical grid-based methods impossible to be applied. 

We demonstrate how to compute dual bounds efficiently by solving the SPDE~\eqref{eq:spde-form}. Our approaches rely on Wong-Zakai-type approximation for the SPDE with a sequence of PDEs. After establishing this approximation, we describe two numerical approaches for solving the resulting PDEs. 

\subsection{Wong-Zakai type approximation of SPDE}
Consider the approximation of the Brownian motion $W$ by a sequence of bounded, continuous, and piecewise differentiable functions $(w_n)_{n\in \mathbb{N}}$. 
As shown in~\cite{brzezniak1995almost}, the solution to an SPDE in Stratonovich form is an almost sure limit of the solutions to a sequence of PDEs in which $W_t$ is replaced by $w_n(t)$. More convergence results for Wong-Zakai approximations can be found in~\cite{twardowska1995approximation, caruana2011rough}. 
For $t\in [0, 1]$, the Karhunen–Lo\`{e}ve expansion of Brownian motion yields a smooth $d$-dimensional approximation $w_n = (w_n^{(j)})_{j=1}^{d}$ with each component taking the form 
\begin{equation*}
    w_n^{(j)}(t) = \sum_{i=0}^{n-1} \frac{\sqrt{2}}{(i+\frac{1}{2})\pi} \xi_{ij} \sin\left( (i+\frac{1}{2})\pi t \right), \quad j=1, \dots, d,
\end{equation*}
where $\xi_{ij} \sim \mathcal{N}(0,1)$ i.i.d. and $n\in \mathbb{N}$. Then the time derivative of the path is given by 
\begin{equation*}
    \dot{w}_n^{(j)}(t) = \sum_{i=0}^{n-1}\sqrt{2}\xi_{ij} \cos\left( (i+\frac{1}{2})\pi t \right), \quad j=1, \dots, d.
\end{equation*}
We collect all the random coefficients in the approximation into the vector $\bm{\xi} := (\xi_{ij})_{i=0,\dots,n-1, j=1,\dots, d} \in \R^{nd}$, which is distributed as a standard $(nd)$-dimensional normal random variable. To emphasize the dependence of the approximate path on these coefficients, we denote $w_n = w_n^{\bm{\xi}}$, and we have 
\begin{equation*}
    \lim_{n\to \infty} w_n^{\bm{\xi}} = W_t \quad \text{ a.s.}
\end{equation*}
When the dependence on $\bm{\xi}$ is clear from the context, we will omit the superscript for notational simplicity. 
Using these approximations, we construct a sequence of PDEs associated with the SPDE~\eqref{eq:spde-form}. 
For each $n\in \mathbb{N}$, the Wong-Zakai-type approximating PDE is 
\begin{equation} \label{eq:wong-zakai-pde}
\begin{aligned}
    \frac{\partial v_n}{\partial t} + \inf_{u\in U}\{\langle \nabla_x v_n, b(x,u)\rangle + r(x, u)\} + \left( \nabla_x v_n^\top \sigma(x) - z_\theta(t,x) \right) \dot{w}_n(t) & \\
    + \frac{1}{2}\op{Tr}(\sigma(x) J_xz_\theta(t,x)) - \langle \nabla_x v_n, \frac{1}{2}(\nabla \sigma:\sigma)(x)\rangle &= 0, \quad (t, x) \in (0, T)\times \R^d \\
    v_n(T, x) &= g(x), \quad x\in \R^d. 
\end{aligned}
\end{equation}

This equation is a first-order HJB equation, and its solution represents the value function of the following deterministic optimal control problem 
\begin{equation} \label{eq:doc-prob}
\begin{aligned}
    v_n(0,x_0) =&\inf_u \bigg\{ \int_0^T \bigg(r(x(t), u(t)) - \langle z_\theta(t, x(t)), \dot{w}_n(t)\rangle + \frac{1}{2}\op{Tr}\left( \sigma(x) J_x z_\theta(t, x(t)) \right)\bigg)\,\dd t + g(x(T))\bigg\}, \\
    &\text{s.t. } \frac{\dd x(t)}{\dd t} = b(x(t), u(t)) - \frac{1}{2}(\nabla \sigma:\sigma)(x) + \sigma(x) \dot{w}_n(t). 
\end{aligned}
\end{equation}
We write $v_n = v_n^{\bm{\xi}}$ to emphasize the dependence of $v_n$ on random coefficients and drop the superscript when the dependence is clear from the context. By~\cite[Section 7]{caruana2011rough}, under suitable conditions, we have 
\begin{equation*}
   v_n^{\bm{\xi}} \overset{p}{\to} v,   
\end{equation*} 
where $v$ is the solution of the SPDE~\eqref{eq:spde-form}. 
The lower bound~\eqref{eq:lower-bound} can be approximated by 
\begin{equation*}
    V_{\text{lower}} = \mathbb{E}[v_\omega(0,x_0)] \approx \mathbb{E}_{\bm{\xi}}[v_n^{\bm{\xi}}(0,x_0)], 
\end{equation*}
for sufficiently large $n\in \mathbb{N}$. 
Finally, we approximate the expectation using the Monte Carlo method by averaging over the solutions corresponding to $M$ independent samples $\bm{\xi}^{(m)}$: 
\begin{equation*}
    \mathbb{E}_{\bm{\xi}}[v_n^{\bm{\xi}}(0,x_0)] \approx \frac{1}{M}\sum_{m=1}^M v_n^{(m)}(0,x_0), 
\end{equation*}
where $v_n^{(m)} = v_n^{\bm{\xi}^{(m)}}$ denotes the solution to the HJ equation corresponding to the $m$-th sample $\bm{\xi}^{(m)}$ of the random vector $\bm{\xi}$. 

\subsection{Two dual approaches}

We now present two methods to solve numerically the dual problem. 

\subsubsection{Method 1: Using Pontryagin's maximum principle}
\label{sec:pontryagin-method}
In this section, we present a curse-of-dimensionality-free method for solving the Hamilton--Jacobi equation \eqref{eq:wong-zakai-pde} (or equivalently \eqref{eq:doc-prob}) for a given realization $w_n = w_n^{\bm{\xi}}$, based on Pontryagin’s maximum principle. The approach characterizes the solution of the associated deterministic optimal control problem through a system of forward–backward ODEs, thereby enabling efficient numerical computation. We begin by introducing the assumptions, following~\cite[Section 3.2]{yong1999stochastic}. Pontryagin’s maximum principle is then stated in Theorem~\ref{thm:pontryagin}. Its proof can be found in~\cite[Section 3.2]{yong1999stochastic}. 

\begin{assumption} \label{assu:Pontryagin}
Let $\bm{\xi}\in \R^{nd}$. Assume that 
\begin{enumerate}
    \item For any parameter $\theta$, the functions $(t,x)\mapsto \langle z_\theta(t,x), \dot{w}_n^{\bm{\xi}}(t)\rangle$ and $(t,x)\mapsto \op{Tr}(\sigma(x)J_xz_\theta(t,x))$ are uniformly continuous with respect to $(t,x)\in [0,T]\times \R^d$. The function $(\nabla \sigma:\sigma)(\cdot)$ is uniformly continuous with respect to $x\in \R^d$.
    \item There exists an $x_0\in \R^d$ and positive constant $A_1, A_2$ such that 
    \begin{equation*}
        \|b(x_0,u)\|_{\R^d} \le A_1, \|r(x_0,u)\|\le A_2  \quad \forall u\in U.
    \end{equation*}
    \item The functions $b(\cdot, u)$, $r(\cdot,u)$, $\sigma$, $g$, $z_\theta(t,\cdot)$, $J_xz_\theta(t,\cdot)$, and $(\nabla \sigma:\sigma)$ are continuously differentiable with respect to $x\in \R^d$ for any $t\in [0,T]$. 
    \item The function $\nabla g$ is uniformly continuous in $x$. $\nabla_x \tilde{b}(t,x,u; w_n^{\bm{\xi}})$ and $\nabla_x \tilde{r}(t,x,u; w_n^{\bm{\xi}})$ are uniformly continuous in $x$ and $u$ for any $t\in [0,T]$, where 
    \begin{equation} \label{eq:defi-tilde-b-r}
    \begin{aligned}
        \tilde{b}(t,x,u; w_n^{\bm{\xi}}) &:= b(x, u) - \frac{1}{2}(\nabla \sigma:\sigma)(x) + \sigma(x) \dot{w}_n^{\bm{\xi}}(t), \\
        \tilde{r}(t,x,u; w_n^{\bm{\xi}}) &:= r(x, u) - \langle z_\theta(t, x), \dot{w}_n^{\bm{\xi}}(t)\rangle + \frac{1}{2}\op{Tr}\left( \sigma(x) J_x z_\theta(t, x) \right). 
    \end{aligned}
    \end{equation}
\end{enumerate}
\end{assumption}

\begin{theorem}[Pontryagin's maximum principle]
\label{thm:pontryagin}
Let $\bm{\xi}\in \R^{nd}$. Let Assumptions \ref{assu:dual-sc-prob} and \ref{assu:Pontryagin} hold. Let $u^*(\cdot)$ be an optimal control for the problem~\eqref{eq:doc-prob} with $w_n = w_n^{\bm{\xi}}$ and $x^*(\cdot)$ be the optimally controlled state trajectory. Then there exists a function $p^*:[0,T] \to \R^d$ such that $(x^*, p^*)$ is a solution pair to the two point boundary value problem 
\begin{equation} \label{eq:two-point-boundary-value-prob}
    \begin{aligned}
        \dot{x}^*(t) &= \partial_p H(t, x^*(t), u^*(t), p^*(t); w_n^{\bm{\xi}}), \quad x^*(0) = x_0, \\
        \dot{p}^*(t) &= -\partial_x H(t,x^*(t), u^*(t), p^*(t); w_n^{\bm{\xi}}), \quad p^*(T) = \nabla g(x^*(T)), 
    \end{aligned}
\end{equation}
where we denote 
\begin{equation} \label{eq:hamiltonian-defi}
\begin{aligned}
    H(t,x,u,p; w_n^{\bm{\xi}}) := &\langle p, b(x,u)\rangle + r(x, u) + \left( p^\top \sigma(x) - z_\theta(t,x) \right) \dot{w}_n^{\bm{\xi}}(t)  \\
    &+ \frac{1}{2}\op{Tr}(\sigma(x) J_xz_\theta(t,x)) - \langle p, \frac{1}{2}(\nabla \sigma:\sigma)(x)\rangle, \\
    \mathcal{H}(t,x,p;w_n^{\bm{\xi}}) :=&\inf_{u\in U} H(t,x,u,p; w_n^{\bm{\xi}}), 
\end{aligned}
\end{equation}
and $u^*$ satisfies 
\begin{equation} \label{eq:u-star}
    H(t, x^*(t), u^*(t), p^*(t); w_n^{\bm{\xi}}) = \inf_{u\in U} H(t, x^*(t), u, p^*(t); w_n^{\bm{\xi}}) \quad \text{ for all } t \in [0,T].
\end{equation}
\end{theorem}

The proposed algorithm can be viewed as a successive approximation scheme for solving the forward-backward ODEs associated with the control problem. Consider a time grid $\mathcal{T} = \{t_i = i\delta t: \delta t = T/N_T, i= 0,1,\dots, N_T\}$. Let $x^{(\ell)}$ and $p^{(\ell)} $ denote the approximate solutions obtained at the $\ell$-th iteration.
At iteration $\ell+1$, the system of forward-backward ODEs~\eqref{eq:two-point-boundary-value-prob} can be discretized using the Euler method as follows: 
\begin{align}
    x_{i+1}^{(\ell+1)} &= x_i^{(\ell+1)} + \tilde{b}(t_i, x_i^{(\ell+1)}, u_i^{(\ell)}; w_n^{\bm{\xi}}) \delta t, \quad x_{0}^{(\ell+1)} = x_0, \label{eq:forward-euler-ode} \\
    p_i^{(\ell + 1)} &= p_{i+1}^{(\ell+1)} + \partial_x H(t_{i+1}, x_{i+1}^{(\ell+1)}, u_{i+1}^{(\ell)}, p_{i+1}^{(\ell+1)}; w_n^{\bm{\xi}})\delta t, \quad p_{N_T}^{(\ell+1)} = \nabla g(x_{N_T}^{(\ell+1)}). \label{eq:backward-euler-ode}
\end{align}
Alternative ODE solvers may also be used to improve computational efficiency.
At each iteration, we first solve the forward equation given the current control, then update the adjoint/backward variable, and finally improve the control via the optimality condition. The update step uses a relaxation (or damping) parameter $\alpha \in (0,1]$, which helps stabilize the fixed-point iteration in practice. Once the stopping criterion is satisfied or the maximum number of iterations is reached, we compute the approximate value by 
\begin{equation} \label{eq:approx-pathwise-value}
    \overline{v}_n^{(m)} = \sum_{i=1}^{N_T} \tilde{r}(t_i, x_i^{(\ell)}, u_i^{(\ell)};w_n^{\bm{\xi}^{(m)}}) \delta t + g(x_{N_T}^{(\ell)}). 
\end{equation}
We summarize the algorithm based on Pontryagin's maximum principle in Algorithm \ref{alg:dual-pontryagin}. 

\begin{algorithm}[htbp!]
\caption{Computing the dual lower bound using Pontryagin's maximum principle}
\label{alg:dual-pontryagin}
\begin{algorithmic}[1]
\REQUIRE The trained neural network approximation $z_\theta$ from any primal approach; a time grid $\mathcal{T}$; number of Monte Carlo samples $M\in \mathbb{N}$; maximum iterations $N_{iter}\in \mathbb{N}$; damping factor $\alpha\in (0,1]$ for updating the adjoint state; initial state $x_0$; convergence threshold $\varepsilon>0$ 
\ENSURE The dual lower bound estimate $\bar{v}_{dual} \approx V_{\text{lower}}$ 
\FOR{$m=1,\dots,M$}
    \STATE Generate a $d\times n$-dimensional standard normal random variable $\bm{\xi}^{(m)}$ 
    \STATE Compute $\dot{w}_n(t) = \dot{w}_n^{\bm{\xi}^{(m)}}(t)$ using $\bm{\xi}^{(m)}$ for $t\in \mathcal{T}$ 
    \STATE Initialize $(x^{(0)}(t), u^{(0)}(t), p^{(0)}(t))$ for $t\in \mathcal{T}$ 
    \FOR{$\ell =0:N_{iter}-1$}
        \STATE Update $x^{(\ell+1)}$ by \eqref{eq:forward-euler-ode} in a forward manner
        \STATE Update $p^{(\ell+1)}$ by \eqref{eq:backward-euler-ode} in a backward manner
        \STATE Update $u^{(\ell+1)}$ and $p^{(\ell+1)}$ by $u^{(\ell+1)} \leftarrow \alpha u^{(\ell+1)} + (1-\alpha)u^{(\ell)}$ 
        \IF{$\|x^{(\ell+1)} - x^{(\ell)}\|_\infty \le  \varepsilon$}
        \BREAK 
        \ENDIF
    \ENDFOR
    \STATE Compute an approximation $\bar{v}_n^{(m)}$ to $v_n^{(m)}(0,x_0)$ by \eqref{eq:approx-pathwise-value} 
\ENDFOR
\RETURN $\bar{v}_{dual} = \frac{1}{M}\sum_{m=1}^M \bar{v}_n^{(m)}$ 
\end{algorithmic}
\end{algorithm}

\subsubsection{Method 2: Using generalized Hopf formula} \label{sec:hopf-method}
In this section, we describe how to solve the Hamilton--Jacobi equation~\eqref{eq:wong-zakai-pde} using the generalized Hopf-Lax formula introduced in~\cite{chow2019algorithm} and also justified in~\cite{yegorov2021perspectives}. The generalized Hopf and Lax formulas provide explicit representations of solutions to HJ equations in terms of maximization and minimization problems, respectively. These formulas extend the classical Hopf–Lax representation, originally developed for state-independent HJ equations, to the state-dependent setting.

In the following, we solve the problem~\eqref{eq:wong-zakai-pde} (equivalently, \eqref{eq:doc-prob}) using the generalized Hopf formula. The resulting maximization formulation is particularly advantageous as it provides a valid lower bound of the optimal value even when the associated optimization problem is solved only approximately. Specifically, assume that we obtain a suboptimal solution $\hat{v}_n^{(m)} \approx v_n^{(m)}(0,x_0)$. There holds $\hat{v}_n^{(m)} \le v_n^{(m)}(0,x_0)$. Then we can compute a true dual lower bound for the original stochastic optimal control problem such that 
\begin{equation*}
    \hat{v}_{\text{dual}} :=\frac{1}{M}\sum_{m=1}^M \hat{v}_n^{(m)} \le \frac{1}{M}\sum_{m=1}^M v_n^{(m)}(0,x_0) \approx \mathbb{E}_{\bm{\xi}}[v_n^{\bm{\xi}}(0,x_0)] \approx \mathbb{E}[v_\omega(0, x_0)] = V_{\text{lower}} \le V(0, x_0),  
\end{equation*}
where the approximations come from the Monte Carlo method and the truncation of the Karhunen–Lo\`{e}ve expansion of the Brownian motion into $n$ terms in the Wong-Zakai type approximation of the SPDE.

The following theorem states the generalized Hopf formula, which was originally proposed as a conjecture in~\cite[eq. (1.2)]{chow2019algorithm}. Subsequent work by Yegorov and Dower identified the proof of this generalized Hopf formula as an open problem, see~\cite[Section 6]{yegorov2021perspectives}. In this paper, we provide a proof under the same assumptions as those required for Pontryagin’s maximum principle, together with the additional assumptions that there exists an optimal control and the terminal cost function $g$ is proper convex. The following assumption and Assumption \ref{assu:Pontryagin} ensure the existence of an optimal control for the problem~\eqref{eq:doc-prob}, see the classical textbook~\cite[Theorem 4.1, Chapter 3]{fleming2012deterministic}. 

\begin{assumption} \label{assu:existence}
    Given any $\bm{\xi} \in \R^{nd}$. Assume that 
    \begin{enumerate}
        \item There exist positive constants $C_1, C_2$ such that 
        \begin{equation*}
            \begin{aligned}
                |\tilde{b}(t,x,u;w_n^{\bm{\xi}})| &\le C_1 (1 +|x| +|u|), \\
                |\tilde{b}(t,x',u;w_n^{\bm{\xi}}) - \tilde{b}(t,x,u;w_n^{\bm{\xi}})| &\le C_2|x-x'|(1 + |u|), 
            \end{aligned}
        \end{equation*}
        for all $t\in [0,T]$, $x, x'\in \R^d$, $u\in U$. 
        \item $U$ is closed. 
        \item For each $(t,x)\in [0,T]\times \R^d$, there exists a continuous function $\ell:U\to \R$ such that 
        \begin{equation*}
            \tilde{r}(t,x,u;w_n^{\bm{\xi}}) \ge \ell(u) \quad \text{ and }\quad \frac{\ell(u)}{|u|} \to +\infty \text{ as } |u|\to \infty. 
        \end{equation*}
        \item For each $(t,x)\in [0,T]\times \R^d$, the set 
        \begin{equation*}
            \tilde{F}(t,x;w_n^{\bm{\xi}}) := \{(z,z_{d+1})\in \R^{d+1}: z=\tilde{b}(t,x,u;w_n^{\bm{\xi}})\in \R^d, z_{d+1} \ge \tilde{r}(t,x,u;w_n^{\bm{\xi}})\in \R, u\in U\}
        \end{equation*}
        is convex. 
    \end{enumerate}
\end{assumption}

\begin{theorem}[Generalized Hopf formula] \label{thm:generalized-hopf}
Given any $\bm{\xi}^{(m)}\in \R^{nd}$. Let Assumptions~\ref{assu:dual-sc-prob}, ~\ref{assu:Pontryagin} and~\ref{assu:existence} hold. Assume that $g$ is proper convex and admits a convex conjugate $g^*$. Let $v_n^{(m)}:[0,T]\times \R^d \to \R$ be the solution to the problem~\eqref{eq:doc-prob} with $w_n = w_n^{\bm{\xi}^{(m)}}$. Then for any $x_0\in \R^d$, the value $v_n^{(m)}(0,x_0)$ can be represented as 
\begin{equation} \label{eq:generalized-hopf-formula}
\begin{aligned}
    v_n^{(m)}(0, x_0) = \sup_{p_0\in \R^d}\bigg\{ &\langle x_0, p_0\rangle - g^*(p(T)) + \int_0^T \bigg[ \mathcal{H}(t, x(t), p(t); w_n^{\bm{\xi}^{(m)}}) - \langle x(t), \partial_x \mathcal{H}(t, x(t), p(t); w_n^{\bm{\xi}^{(m)}})\rangle \bigg]\,\dd t: \\
    & \dot{x}(t) = \partial_p \mathcal{H}(t, x(t), p(t); w_n^{\bm{\xi}^{(m)}}), \quad x(0) = x_0, \\
    & \dot{p}(t) = -\partial_x \mathcal{H}(t,x(t),p(t); w_n^{\bm{\xi}^{(m)}}), \quad p(0)= p_0 \bigg\},  
\end{aligned}
\end{equation}
where $\mathcal{H}$ is defined in \eqref{eq:hamiltonian-defi}. 
\end{theorem}
\begin{proof}
{\bf Step 1.} We begin by characterizing the optimal cost using Pontryagin's Maximum Principle (Theorem \ref{thm:pontryagin}). The value function is given by 
\begin{equation} \label{eq:optimal-cost}
    v_n^{(m)}(0,x_0) = g(x^*(T)) + \int_0^T \tilde{r}(t,x^*(t), u^*(t);w_n^{\bm{\xi}^{(m)}})\,\dd t, 
\end{equation}
where the optimal trajectory $(x^*, p^*)$ solves the Hamiltonian system 
\begin{equation} \label{eq:Pontryagin-ode}
\begin{aligned}
    \dot{x}^*(t) &= \partial_p \mathcal{H}(t,x^*(t), p^*(t); w_n^{\bm{\xi}^{(m)}}), \quad &x^*(0) = x_0, \\
    \dot{p}^*(t) &= -\partial_x \mathcal{H}(t, x^*(t), p^*(t); w_n^{\bm{\xi}^{(m)}}), \quad &p^*(T) = \nabla g(x^*(T)). 
\end{aligned}
\end{equation}
Furthermore, by Assumption \ref{assu:existence} (existence of optimal control), for any $t\in [0,T]$ there exists $u^*(t)$ satisfying the minimization condition 
\begin{equation*}
    u^*(t) \in \operatorname*{arg\,inf}_{u\in U} H(t,x^*(t), u, p^*(t); w_n^{\bm{\xi}^{(m)}}).
\end{equation*}
We shall express the running cost $\tilde{r}$ in terms of the Hamiltonian and the state dynamics. 
Recalling the definition of $\mathcal{H}$ and observing from \eqref{eq:Pontryagin-ode} that $\dot{x}^*(t) = \tilde{b}(t,x^*(t), u^*(t); w_n^{\bm{\xi}^{(m)}})$, we have 
\begin{equation*}
    \mathcal{H}(t,x^*(t), p^*(t); w_n^{\bm{\xi}^{(m)}}) = \langle p^*(t), \tilde{b}(t,x^*(t), u^*(t); w_n^{\bm{\xi}^{(m)}})\rangle + \tilde{r}(t,x^*(t), u^*(t); w_n^{\bm{\xi}^{(m)}}). 
\end{equation*}
Rearranging terms yields the relationship 
\begin{equation} \label{eq:r-tilde-by-hamiltonian}
    \tilde{r}(t,x^*(t), u^*(t); w_n^{\bm{\xi}^{(m)}}) = \mathcal{H}(t,x^*(t), p^*(t); w_n^{\bm{\xi}^{(m)}}) - \langle p^*(t), \dot{x}^*(t) \rangle. 
\end{equation}

{\bf Step 2.} We establish the following identity for any $t\in [0,T]$, 
\begin{equation} \label{eq:equivalence-1}
    \mathcal{H}(t,x^*(t), p^*(t); w_n^{\bm{\xi}^{(m)}}) - \langle p^*(t), \dot{x}^*(t) \rangle = \sup_{p\in \R^d}\left\{ \mathcal{H}(t,x^*(t), p; w_n^{\bm{\xi}^{(m)}}) - \langle p, \dot{x}^*(t) \rangle \right\}.
\end{equation}
First, observe that for any fixed control $u\in U$, the mapping $p\mapsto H(t,x^*(t),u, p;w_n^{\bm{\xi}^{(m)}})$ is affine. Since $\mathcal{H}$ is defined as the infimum of these affine functions over $u$, the map $p\mapsto \mathcal{H}(t,x^*(t),p;w_n^{\bm{\xi}^{(m)}})$ is concave. 

Then, we define the convex function 
\begin{equation*}
    h(p) := -\mathcal{H}(t,x^*(t),p;w_n^{\bm{\xi}^{(m)}}). 
\end{equation*}
Its convex conjugate $h^*$ is given by 
\begin{equation*}
    h^*(q) := \sup_{p\in \mathbb{R}^d}\left\{ \langle p, q\rangle - h(p) \right\} = \sup_{p\in \R^d}\left\{ \langle p, q\rangle +\mathcal{H}(t,x^*(t), p;w_n^{\bm{\xi}^{(m)}}) \right\}
\end{equation*}
Evaluating $h^*$ at $q = -\dot{x}^*(t)$, we have 
\begin{equation}
\label{eq:hstar-xstar}
    h^*(-\dot{x}^*(t)) = \sup_{p\in \R^d} \left\{ \mathcal{H}(t,x^*(t), p;w_n^{\bm{\xi}^{(m)}}) - \langle p, \dot{x}^*(t)\rangle \right\}. 
\end{equation}
By the Fenchel-Young inequality, for any $p\in \R^d$, 
\begin{equation*}
    h^*(-\dot{x}^*(t)) + h(p) \ge \langle p, -\dot{x}^*(t) \rangle, 
\end{equation*}
with equality holding if and only if $-\dot{x}^*(t) =  -\partial_p \mathcal{H}(t,x^*(t),p;w_n^{\bm{\xi}^{(m)}})$. 
Recalling the state dynamics from Step 1, we have $\dot{x}^*(t) =  \partial_p \mathcal{H}(t,x^*(t),p;w_n^{\bm{\xi}^{(m)}})$. 
Thus, the equality condition is satisfied at $p = p^*(t)$ and we have 
\begin{equation*}
    h^*(-\dot{x}^*(t)) + h(p^*(t)) = \langle p^*(t), -\dot{x}^*(t) \rangle.
\end{equation*}
Rearranging this equality and substituting the definitions of $h$ and \eqref{eq:hstar-xstar} recovers equation~\eqref{eq:equivalence-1}. 

{\bf Step 3.} We now establish a lower bound for the optimal cost by passing to a dual formulation.
Since $g$ is proper convex and continuous by Assumption~\ref{assu:dual-sc-prob}, we express the terminal cost using its convex conjugate 
\begin{equation} \label{eq:conjugate-g}
    g(x^*(T)) = \sup_{p\in \R^d} \left\{ \langle x^*(T), p \rangle - g^*(p) \right\}. 
\end{equation}
Substituting \eqref{eq:conjugate-g} and the identity \eqref{eq:equivalence-1} from Step 2 into the cost equation \eqref{eq:optimal-cost}, we obtain 
\begin{equation*}
    v_n^{(m)}(0,x_0) = \sup_{p\in \R^d} \left\{ \langle x^*(T), p \rangle - g^*(p) \right\} + \int_0^T \sup_{p\in \R^d}\left\{ \mathcal{H}(t,x^*(t), p; w_n^{\bm{\xi}^{(m)}}) - \langle p, \dot{x}^*(t) \rangle \right\} \,\dd t 
\end{equation*}
We now relax the pointwise supremum inside the integral to a supremum over the space of measurable functions $p:[0,T]\to \R^d$: 
\begin{equation} \label{eq:sup-integral-inequality} 
v_n^{(m)}(0,x_0) \ge \sup_{\substack{p:[0,T]\to \R^d \\ \text{measurable}}} \left\{ \langle x^*(T), p(T) \rangle - g^*(p(T)) + \int_0^T \left[ \mathcal{H}(t,x^*(t), p(t); w_n^{\bm{\xi}^{(m)}}) - \langle p(t), \dot{x}^*(t) \rangle\right] \,\dd t \right\}. 
\end{equation}
Here, we make some remark on measurability. The interchange of supremum and integral is justified as follows. The map $(t,x,p,u)\mapsto \langle \tilde{b}(t,x,u;w_n^{\bm{\xi}^m}), p\rangle + \tilde{r}(t,x,p;w_n^{\bm{\xi}^m})$ is continuous by Assumption~\ref{assu:Pontryagin} and thus measurable. Since $U$ is a separable metric space, the infimum defining $\mathcal{H}$ can be taken over a countable dense subset, preserving measurability. Consequently, the map $t\mapsto \sup_{p\in \R^d}\{\mathcal{H}(t, x^*(t), p;w_n^{\bm{\xi}^{(m)}}) - \langle p, \dot{x}^*(t)\rangle\}$ is measurable. Its integrability is guaranteed by the integrability of $\tilde{r}$ and the equality derived in \eqref{eq:r-tilde-by-hamiltonian} and \eqref{eq:equivalence-1}. Moreover, the measurability of $t\mapsto \mathcal{H}(t,x^*(t), p(t); w_n^{\bm{\xi}^{(m)}}) - \langle p(t), \dot{x}^*(t) \rangle$ is ensured since we restrict $p:[0,T]\to \R^d$ to be measurable. 

Next, applying integration by parts to the term involving $\dot{x}^*(t)$ yields 
\begin{equation} \label{eq:optimize-over-all-p}
\begin{aligned}
    v_n^{(m)}(0,x_0) &\ge \sup_{\substack{p:[0,T]\to \R^d \\ \text{measurable}}} \bigg\{ x^*(T), p(T) \rangle - g^*(p(T)) + \int_0^T \mathcal{H}(t,x^*(t), p(t); w_n^{\bm{\xi}^{(m)}})\,\dd t \\
    &\qquad \qquad \qquad - \langle x^*(T), p(T)\rangle + \langle x_0, p(0) \rangle + \int_0^T \langle x^*(t), \dot{p}(t) \rangle \,\dd t \bigg\} \\
    &= \sup_{\substack{p:[0,T]\to \R^d \\ \text{measurable}}} \bigg\{\langle x_0, p(0) \rangle - g^*(p(T)) + \int_0^T \left(\mathcal{H}(t,x^*(t), p(t); w_n^{\bm{\xi}^{(m)}}) + \langle x^*(t), \dot{p}(t) \rangle\right) \,\dd t \bigg\}. 
\end{aligned}
\end{equation}
Then, we intend to optimize over a subset of $\{p:[0,T]\to \R^d\mid p \text{ is measurable}\}$. Define the set $S_p^{x_0}$ of continuous functions $p(\cdot)$ that are solutions to the Hamiltonian system by 
\begin{equation*}
S_p^{x_0} := \left\{
p\in C([0,T]; \mathbb{R}^d) \;\bigg|\; 
\exists x: [0,T] \to \mathbb{R}^d \text{ s.t. }
\begin{array}{l}
\dot{p}(t) = -\partial_x \mathcal{H}(t,x(t), p(t); w_n^{\bm{\xi}^{(m)}}), \\
\dot{x}(t) = \partial_p\mathcal{H}(t,x(t), p(t); w_n^{\bm{\xi}^{(m)}}), \quad x(0) = x_0
\end{array}
\right\}.
\end{equation*}
Here, $p(0)$ is ``free". By Assumption~\ref{assu:existence} (existence of optimal control) and Theorem~\ref{thm:pontryagin}, there exists an optimal trajectory $(x^*, p^*)$ such that $p^* \in S_p^{x_0}$ and thus $S_p^{x_0}$ is nonempty.
Restricting the domain of the supremum in \eqref{eq:optimize-over-all-p} to $S_p^{x_0}$ yields the lower bound: 
\begin{equation} \label{eq:quantity-I}
    v_n^{(m)}(0,x_0) \ge \sup_{p\in S_p^{x_0}} \bigg\{\langle x_0, p(0) \rangle - g^*(p(T)) + \int_0^T \left(\mathcal{H}(t,x^*(t), p(t); w_n^{\bm{\xi}^{(m)}}) + \langle x^*(t), \dot{p}(t) \rangle\right) \,\dd t \bigg\} =: \textrm{I}. 
\end{equation}

{\bf Step 4.} Next, we will show that 
\begin{equation} \label{eq:identity-I-II}
    \textrm{I} \ge \textrm{II} := \sup_{(x,p)\in S_{xp}^{x_0}} \bigg\{\langle x_0, p(0) \rangle - g^*(p(T)) + \int_0^T \left(\mathcal{H}(t,x(t), p(t); w_n^{\bm{\xi}^{(m)}}) + \langle x(t), \dot{p}(t) \rangle\right) \,\dd t \bigg\}, 
\end{equation}
where the admissible set $S_{xp}^{x_0}$ is defined by 
\begin{equation*}
    S_{xp}^{x_0} := \left\{
(x,p)\in C([0,T]; \mathbb{R}^d)^2 \;\bigg|\;
\begin{array}{l}\dot{p}(t) = -\partial_x \mathcal{H}(t,x(t), p(t); w_n^{\bm{\xi}^{(m)}}), \\
\dot{x}(t) = \partial_p\mathcal{H}(t,x(t), p(t); w_n^{\bm{\xi}^{(m)}}),\quad x(0) = x_0
\end{array}\right\}. 
\end{equation*}
For any $(x, p) \in S_{xp}^{x_0}$, it holds that $p\in S_p^{x_0}$. Therefore, to establish \eqref{eq:identity-I-II}, it suffices to prove that the integrand satisfy the pointwise inequality 
\begin{equation} \label{eq:goal-ineq}
    \mathcal{H}(t, x^*(t), p(t); w_n^{\bm{\xi}^{(m)}}) + \langle x^*(t), \dot{p}(t)\rangle \ge \mathcal{H}(t,x(t), p(t); w_n^{\bm{\xi}^{(m)}}) + \langle x(t), \dot{p}(t) \rangle. 
\end{equation}
Define the convex conjugate of $\mathcal{H}$ with respect to $x$ by 
\begin{equation*}
    F(\eta; t, p(t), w_n^{\bm{\xi}^{(m)}}) := \sup_{x\in \R^d} \left\{ \langle \eta, x\rangle - \mathcal{H}(t, x, p(t); w_n^{\bm{\xi}^{(m)}}) \right\}. 
\end{equation*}
By the Fenchel–Young inequality, evaluating the conjugate at $\eta = -\dot{p}(t)$ and the primal function at $x^*(t)$ yields 
\begin{equation} \label{eq:FY-1}
    \mathcal{H}(t, x^*(t), p(t); w_n^{\bm{\xi}^{(m)}}) + F(-\dot{p}(t); t, p(t), w_n^{\bm{\xi}^{(m)}}) \ge \langle x^*(t), -\dot{p}(t)\rangle. 
\end{equation}
Furthermore, the Fenchel-Young inequality becomes an equality if $\eta = \partial_x\mathcal{H}$. Since $(x, p) \in S_{xp}^{x_0}$ implies that $-\dot{p}(t) = \partial_x\mathcal{H}(t, x(t), p(t); w_n^{\bm{\xi}^{(m)}})$, we obtain the identity 
\begin{equation} \label{eq:FY-2}
    \mathcal{H}(t, x(t), p(t); w_n^{\bm{\xi}^{(m)}}) + F(-\dot{p}(t); t, p(t), w_n^{\bm{\xi}^{(m)}}) = \langle x(t), -\dot{p}(t)\rangle. 
\end{equation}
Combining \eqref{eq:FY-1} and \eqref{eq:FY-2} leads to the pointwise inequality \eqref{eq:goal-ineq}, which in turn confirms $\textrm{I} \ge \textrm{II}$.
Therefore, by \eqref{eq:quantity-I} and \eqref{eq:identity-I-II}, we obtain 
\begin{equation} \label{eq:generalized-hopf-one-side}
\begin{aligned}
    v_n^{(m)}(0, x_0) \ge \sup_{p_0\in \R^d}\bigg\{ &\langle x_0, p_0\rangle - g^*(p(T)) + \int_0^T \bigg[ \mathcal{H}(t, x(t), p(t); w_n^{\bm{\xi}^{(m)}}) + \langle x(t), \dot{p}(t)\rangle \bigg]\,\dd t: \\
    & \dot{x}(t) = \partial_p \mathcal{H}(t, x(t), p(t); w_n^{\bm{\xi}^{(m)}}), \quad x(0) = x_0, \\
    & \dot{p}(t) = -\partial_x \mathcal{H}(t,x(t),p(t); w_n^{\bm{\xi}^{(m)}}), \quad p(0)= p_0 \bigg\}. 
\end{aligned}
\end{equation}

{\bf Step 5.} Finally, it remains to show that the inequality in \eqref{eq:generalized-hopf-one-side} is in fact an equality. We use the Fenchel-Young inequality 
\begin{equation*}
    \langle x(T), p(T)\rangle \le g(x(T)) + g^*(p(T)), 
\end{equation*}
where the equality holds only when $p(T) = \nabla g(x(T))$, since $g$ is continuously differentiable by Assumption~\ref{assu:Pontryagin}. By \eqref{eq:Pontryagin-ode}, we have $(x^*, p^*)\in S_{xp}^{x_0}$ and $p^*(T) = \nabla g(x^*(T))$.
Thus, there holds 
\begin{equation} \label{eq:equality-g-g-star}
    \langle x^*(T), p^*(T)\rangle = g(x^*(T)) + g^*(p^*(T)). 
\end{equation}
Using \eqref{eq:generalized-hopf-one-side}, \eqref{eq:equality-g-g-star}, and integration by parts, we obtain 
\begin{equation*}
    \begin{aligned}
        v_n^{(m)}(0,x_0) &\ge \langle x_0, p^*(0)\rangle - g^*(p^*(T)) + \int_0^T \bigg[ \mathcal{H}(t, x^*(t), p^*(t); w_n^{\bm{\xi}^{(m)}}) + \langle x^*(t), \dot{p}^*(t)\rangle \bigg]\,\dd t \\
        &= \langle x_0, p^*(0)\rangle - g(x^*(T)) - \langle x^*(T), p^*(T)\rangle + \int_0^T \bigg[\mathcal{H}(t, x^*(t), p^*(t); w_n^{\bm{\xi}^{(m)}})\,\dd t + \langle x^*(t), \dot{p}^*(t)\rangle \bigg]\,\dd t \\
        &= g(x^*(T)) + \int_0^T \bigg[\mathcal{H}(t, x^*(t), p^*(t); w_n^{\bm{\xi}^{(m)}})\,\dd t + \langle p^*(t), \dot{x}^*(t)\rangle \bigg]\,\dd t \\
        &= v_n^{(m)}(0,x_0),
    \end{aligned}
\end{equation*}
where the last equality holds due to \eqref{eq:optimal-cost} and \eqref{eq:r-tilde-by-hamiltonian}. 
This verifies that \eqref{eq:generalized-hopf-one-side} is an equality and thus completes the proof. 
\end{proof}

\begin{remark}
For clarity of presentation, we formulated Theorem~\ref{thm:generalized-hopf} and its associated assumptions specifically for the problem~\eqref{eq:doc-prob}, as the primary goal of this paper is to develop a dual numerical method. However, the proof also extends to general deterministic optimal control problems provided that an optimal control exists, Pontryagin’s maximum principle holds, and the terminal cost $g$ is proper convex. 
\end{remark}

We summarize the dual approach based on the generalized Hopf formula in Algorithm \ref{alg:dual-generalized-hopf}. 

\begin{algorithm}[htbp!]
\caption{Computing the dual lower bound using the generalized Hopf formula}
\label{alg:dual-generalized-hopf}
\begin{algorithmic}[1]
\REQUIRE The trained neural network approximation $z_\theta$ from any primal approach; a time grid $\mathcal{T}$; number of Monte Carlo samples $M\in \mathbb{N}$; optimization iterations $N_{opt}\in \mathbb{N}$; learning rate $\eta$; the initial state $x_0$ 
\ENSURE The dual lower bound estimate $\hat{v}_{dual} \approx V_{\text{lower}}$ 
\FOR{$m=1, \dots, M$}
    \STATE Generate a $d\times n$-dimensional standard normal random variable $\bm{\xi}^{(m)}$ 
    \STATE Compute $\dot{w}_n(t) = \dot{w}_n^{\bm{\xi}^{(m)}}(t)$ using $\bm{\xi}^{(m)}$ for $t\in \mathcal{T}$ 
        \STATE Initialize $p_0 \in \mathbb{R}^d$ randomly 
        \STATE Initialize  
        \begin{align*}
            \hat{v}_n^{(m)} \leftarrow \mathcal{J}(p_0) &:= \langle x_0, p_0\rangle - g^*(x(T)) 
            \\
            &\qquad + \sum_{t\in \mathcal{T}} \bigg[\mathcal{H}(t, x(t), p(t); w_n^{\bm{\xi}^{(m)}}) - \langle x(t), \partial_x\mathcal{H}(t, x(t), p(t); w_n^{\bm{\xi}^{(m)}})\rangle \bigg]\Delta t,
        \end{align*}
        where $(x, p)$ is solved with the initial condition $(x_0, p_0)$ using an ODE solver 
        \FOR{$\ell = 1, \dots,N_{opt}$}
            \STATE Compute gradient $\nabla_{p_0}\mathcal{J}(p_0)$ by automatic differentiation 
            \STATE Update $\hat{v}_n^{(m)} \leftarrow \mathcal{J}(p_0)$ 
            \STATE Update $p_0 \leftarrow p_0 +\eta \nabla_{p_0}\mathcal{J}(p_0)$  (Gradient ascent)
        \ENDFOR
\ENDFOR
\RETURN $\hat{v}_{dual} = \frac{1}{M}\sum_{m=1}^M \hat{v}_n^{(m)}$. 
\end{algorithmic}
\end{algorithm}

\begin{remark}
In Algorithm~\ref{alg:dual-generalized-hopf}, we solve the maximization problem in the generalized Hopf formula \eqref{eq:generalized-hopf-formula} using gradient ascent, facilitated by automatic differentiation. Importantly, due to the maximization structure of the formula, the resulting value remains a valid lower bound (a true dual bound) even if the optimization procedure only yields a sub-optimal solution. This, together with the upper bound, can be used to assess the quality of approximate solutions. 
\end{remark}

\section{Numerical examples}
\label{sec:numerical-examples}
In this section, we provide numerical experiments to compute the lower and upper bounds of the stochastic control problem~\eqref{eq:sc-prob}. The algorithms are implemented in Python, where primal methods are implemented using the machine learning library PyTorch and dual methods are implemented using the Jax library for computational efficiency. The codes for Algorithm~\ref{alg:dual-pontryagin} and ~\ref{alg:dual-generalized-hopf} can be found in the public GitHub repository \href{https://github.com/jiefeiy/dual-stochastic-optimal-control/tree/main}{https://github.com/jiefeiy/dual-stochastic-optimal-control/tree/main}. 
To justify Assumption~\ref{assu:Pontryagin}, the model $z_\theta(t,\cdot)$ obtained from any primal method should be continuously differentiable. Thus, in numerical experiments, we adopt $\tanh$ activation function instead of the popular choice \textrm{ReLU} in neural networks. 

\subsection{Linear quadratic stochastic control}
We consider a $d$-dimensional linear quadratic problem. Let $T=1$, $x_0 = (0,\dots,0)^\top$, $A = \op{diag}([1, 4/25,4,\dots,4])$, and define 
\begin{equation*}
    \begin{aligned}
        r(x,u) &= |u|^2, \quad g(x) = \frac{1}{2}(x^\top A x + 1), \\
        b(x,u) &= 2u, \quad \sigma(x) = \sqrt{2}. 
    \end{aligned}
\end{equation*}
The associated value function satisfies the HJB equation 
\begin{equation*}
\begin{aligned}
    \frac{\partial V}{\partial t} + \Delta_x V - |\nabla_x V|^2 &= 0,\quad (t,x)\in [0,T)\times \R^d, \\
    V(T,x) &= g(x), \quad x\in \R^d.
\end{aligned}
\end{equation*}
By the Hopf-Cole transformation, this HJB equation has an explicit solution. In particular, the value function at the initial time $t=0$ is given by 
\begin{equation*}
    V(0, x_0) = -\ln \bE\left[ \exp\left( -g(x_0 + \sqrt{2}W_T) \right) \right],
\end{equation*}
where $(W_t)_{t\ge 0}$ is a standard $d$-dimensional Brownian motion. This closed-form expression will serve as a benchmark against which we compare the lower and upper bounds produced by the proposed primal-dual approaches. 

We implement the proposed primal-dual algorithms across dimensions $d\in \{2,3,5,10\}$. The results are summarized in Table~\ref{tab:bounds-lq-dim-Pontryagin}. The lower bounds are computed using Pontryagin's maximum principle and the generalized Hopf formula. For both dual approaches, we employ a time discretization of $N_T=200$, Brownian motion approximation with $n=32$ terms, the number of Monte Carlo samples $M=500$, together with a fourth-order Runge--Kutta ODE solver. The upper bounds are obtained via the deep BSDE method using a three-layer feedforward neural network with a hidden dimension of 64 and a $\tanh$ activation function. The evaluation is performed using $N_{T,\mathrm{upper}}=400$ time steps and a sample size of $M_{\mathrm{upper}}=2^{18}$.

\begin{table}[htbp!]
    \centering
    \begin{tabular}{ccccc}
        \hline 
        $d$ & $V_{\text{lower}}$ (Pontryagin) & $V_{\text{lower}}$ (Hopf) & $V(0,x_0)$ & $V_{\text{upper}}$ (Deep BSDE) \\
        \hline 
        2 & $1.1858 \pm 0.0018$ & $1.1889 \pm 0.0016$ & 1.1881 & $1.1891 \pm 0.0026$  \\
        3 & $2.2579 \pm 0.0122$ & $2.2845 \pm 0.0113$ & 2.2868 & $2.3018 \pm 0.0052$  \\
        5 & $4.3780 \pm 0.0224$ & $4.4513 \pm 0.0207$ & 4.4841 & $4.5286 \pm 0.0084$  \\
        10 & $9.5273 \pm 0.0596$ & $10.0452 \pm 0.1270$ & 9.9759 & $10.3618 \pm 0.0143$ \\
        \hline 
    \end{tabular}
    \caption{Numerical comparison of lower and upper bounds for $V(0,x_0)$ in different dimensions $d$. Reported values include $95\%$ confidence intervals.}
    \label{tab:bounds-lq-dim-Pontryagin}
\end{table}

Next we test the primal-dual algorithm for different initial states $x_0=(x_0^{(1)}, 0)$ with $x_0^{(1)} \in \{-2, -1.5, ..., 2\}$, see Figure~\ref{fig:lq_figure_x0}. The lower bounds are computed using the generalized Hopf formula, maintaining the same hyperparameters as in the previous experiment. The upper bounds are obtained via the deep BSDE method using a three-layer feedforward neural network with a hidden dimension of $64$. For this set of experiments, we use $N_{T,\mathrm{upper}}=200$ time steps and $M_{\mathrm{upper}}=2^{17}$ samples for the upper bound evaluation. In Figure~\ref{fig:lq_figure_x0}(a), all exact values fall within the corresponding $95\%$ confidence intervals of the estimated bounds. Figure~\ref{fig:lq_figure_x0}(b) plots the relative duality gap $\epsilon$. We observe that the relative error remains roughly of order $10^{-3}$ across varying initial states.

\begin{figure}[htbp!]
    \centering
    \begin{subfigure}{.45\textwidth}
        \centering
        \includegraphics[width=\linewidth]{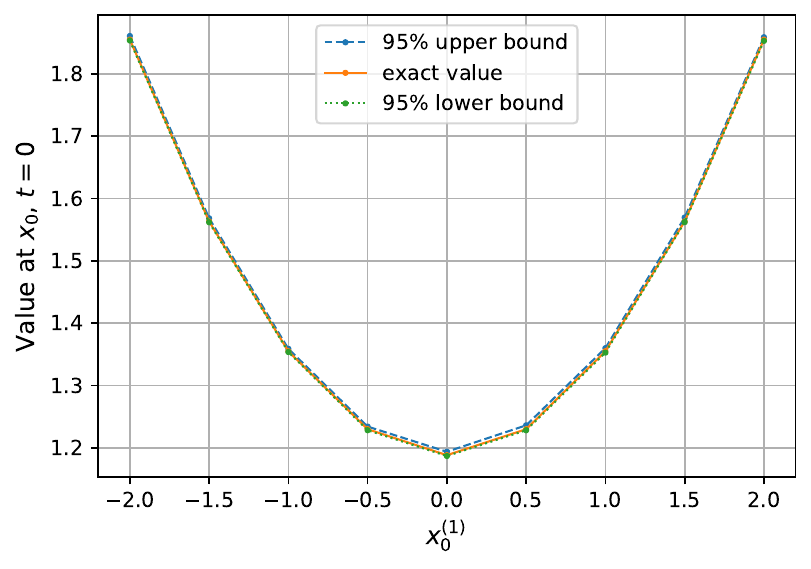}
        \caption{Lower and upper bounds of $V(0,x_0)$}
    \end{subfigure}
    \begin{subfigure}{.45\textwidth}
        \centering
        \includegraphics[width=\linewidth]{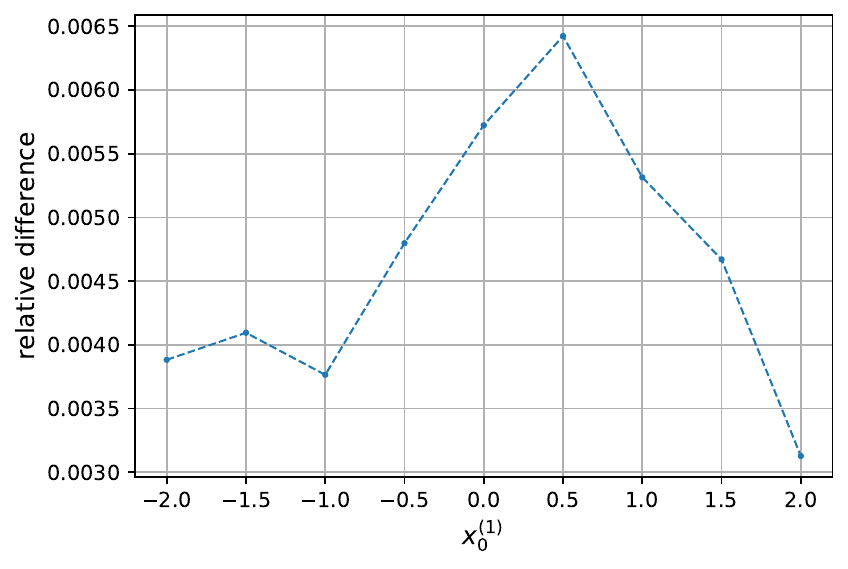}
        \caption{Relative difference $\epsilon = \frac{V_{\text{upper}} - \hat{v}_{dual}}{V(0,x_0)}$}
    \end{subfigure}
    \caption{$95\%$ confidence intervals' bounds and relative difference of $V(0,x_0)$ with $x_0=(x_0^{(1)},0)$.}
    \label{fig:lq_figure_x0}
\end{figure}

\subsection{Ornstein-Uhlenbeck dynamics with linear costs}
Consider a controlled Ornstein–Uhlenbeck process with linear–quadratic costs, as studied in~\cite[Section 6.2]{nusken2021solving}. Let the running cost, terminal cost, drift, and diffusion coefficients be 
\begin{equation*}
\begin{aligned}
    r(x,u) = \frac{1}{2}|u|^2, \quad g(x) = \gamma \cdot x, \\
    b(x,u) = Ax + Bu, \quad \sigma(x) = B, 
\end{aligned}
\end{equation*}
where $\gamma \in \R^d$ is a fixed vector, and $A,B\in \R^{d\times d}$ are constant matrices. The associated value function satisfies the HJB PDE 
\begin{equation*}
    \partial_t V +\frac{1}{2}\op{Tr}(BB^\top \op{Hess}_x V) + \inf_u \left\{ \langle \nabla_x V, Ax + Bu\rangle + \frac{1}{2}|u|^2\right\} = 0 
\end{equation*}
with terminal condition $V(T, x) = \gamma \cdot x$. The minimization over $u$ yields the optimal feedback control $u^*_t = -B^\top\nabla_xV$ and the HJB equation can be written equivalently as 
\begin{equation*}
    \partial_t V +\frac{1}{2}\op{Tr}(BB^\top \op{Hess}_x V) +\langle \nabla_x V, Ax\rangle - \frac{1}{2}|B^\top \nabla_x V|^2 = 0. 
\end{equation*}

In numerical experiments, we take $A = -I_{d\times d} + (c_{ij}^A)_{1\le i,j\le d}$, $B = I_{d\times d} +(c_{i,j}^B)_{1\le i,j\le d}$ with $c_{ij}^A = 0.12$, $c_{ij}^B = 0.1$, and set $\gamma = (1,\dots,1)^\top$, $T=1$. 
For this choice of parameters, the value function admits the explicit expression 
\begin{equation*}
    V(0,x_0) = e^{\lambda_A} \sum_{i=1}^d x_0^{(i)} - \frac{d \lambda_B^2}{4 \lambda_A}(e^{2\lambda_A} - 1), \quad \text{ with } \lambda_A = 0.12d - 1 , \lambda_B = 1+0.1d, 
\end{equation*}
which we use as a reference solution. 

We solve the problem numerically using a primal–dual algorithm based on Pontryagin’s maximum principle; the results are reported in Table~\ref{tab:bounds-ou-Pontryagin}. To apply the generalized Hopf formula, we first compute the convex conjugate of the terminal cost: 
\begin{equation*}
    g^*(p) = \sup_{x\in \R^d} \left\{ \langle p - \gamma, x\rangle \right\} = \begin{cases}
        0, &p = \gamma, \\
        +\infty, &p \ne \gamma. 
    \end{cases}
\end{equation*}
Consequently, the integral term in the generalized Hopf formula~\eqref{eq:generalized-hopf-formula} is finite only when $p(T)=\gamma$. Maximizing the objective therefore enforces the terminal condition $p(T)=\gamma$, which coincides with the transversality condition $p(T)=\nabla g(x)=\gamma$ in Pontryagin’s maximum principle. This confirms the consistency between the generalized Hopf formulation and the Pontryagin-based approach. 

\begin{table}[htbp!]
    \centering
    \begin{tabular}{ccccc}
        \hline 
        $d$ & $V_{\text{lower}}$ (Pontryagin) & $V(0,x_0)$ & $V_{\text{upper}}$ (Deep BSDE) & $95\%$ confidence interval \\
        \hline 
        2 & $0.1914 \pm 0.0015$ & $0.1952$ & $0.1899 \pm 0.0066$ & $[0.1900, 0.1965]$ \\
        3 & $0.1499 \pm 0.0019$ & $0.1521$ & $0.1500 \pm 0.0092$ & $[0.1481, 0.1592]$ \\
        5 & $-0.5249 \pm 0.0021$ & $-0.5203$ & $-0.5168 \pm 0.0151$ & $[-0.5270, -0.5017]$ \\
        10 & $-12.6287 \pm 0.1594$ & $-12.3772$ & $-11.7802 \pm 0.0359$ & $[-12.7880, -11.7444]$ \\
        \hline 
    \end{tabular}
    \caption{Numerical comparison of lower and upper bounds for $V(0,x_0)$ with $x_0 = [1, \dots, 1]^\top$}
    \label{tab:bounds-ou-Pontryagin}
\end{table}

For experiments reported in Table~\ref{tab:bounds-ou-Pontryagin}, the upper bounds are obtained via the deep BSDE method using a three-layer feedforward neural network with a hidden dimension of 64 and a $\tanh$ activation function, while the lower bounds are computed with $N_T=400, n=32, M=500$. 
Across all dimensions, the true value falls within the $95\%$ confidence intervals. For lower dimensions ($d=2, 3$), we observe that the empirical mean of the Deep BSDE upper bound is slightly below the true value due to statistical noise. The duality gap widens slightly as the state dimension increases. Nonetheless, the bounds remain tight enough to provide a meaningful certificate of near-optimality even in 10-dimensional space. 
Regarding computational efficiency, for the high-dimensional case ($d=10$), we trained the primal method in 995 seconds and the dual method completed execution in 152 seconds performed on a laptop with a $2.50$ GHz Intel Core i9-12900H processor and $32$ GB of RAM. 

\subsection{Aiyagari's growth model in economics}
In this section, we compute the dual bounds combined with deep actor-critic method for the Aiyagari's growth model in economics~\cite{zhou2024solving}. The agent aims to control the dynamic to maximize the expected utility 
\begin{equation*}
    \max_{u}\bE\left[ \int_0^T (U(u_t) - \bar{r}(Z_t))\,\dd t + A_T - \bar{g}(Z_T) \right],
\end{equation*}
where the controlled dynamics are 
\begin{equation*}
    \begin{aligned}
        \dd Z_t &= -(Z_t - 1)\,\dd t + \sigma_z\,\dd W_t^{(1)}, \\
        \dd A_t &= ((1-\alpha)Z_t + (\alpha - \delta)A_t - u_t) \,\dd t + \sigma_a A_t \,\dd W_t^{(2)}. 
    \end{aligned}
\end{equation*}
We take the same parameters as in~\cite[Section 6.2]{zhou2024solving}. Let $U(u) = \log(u)$, $\bar{g}(z) = 0.2\exp(0.2z)$, $\alpha = \delta = 0.05$, $\sigma_z = 1.0$, $\sigma_a = 0.1$, $T=0.1$. 

To be consistent with the problem~\eqref{eq:sc-prob}, taking a change of sign, we consider the equivalent minimization problem: 
\begin{equation*}
    V(0,Z_0, A_0) =\min_u \bE\left[ \int_0^T (-U(u_t) + \bar{r}(Z_t))\,\dd t + \bar{g}(Z_T) - A_t \right].
\end{equation*}
Let $(Z_t, A_t) = (X_t^{(1)}, X_t^{(2)})$. The drift, diffusion, running cost, and terminal cost function are given by 
\begin{equation*}
    b(x,u) = \begin{bmatrix}
        1-x^{(1)} \\
        0.95x^{(1)} - u
    \end{bmatrix}, \quad \sigma(x) = \begin{bmatrix}
        1 & 0 \\
        0 & 0.1x^{(2)}
    \end{bmatrix}.
\end{equation*}
\begin{equation*}
    r(x,u) = -\log(u) + \bar{r}(x^{(1)}), \quad g(x) = \bar{g}(x^{(1)})- x^{(2)}. 
\end{equation*}
where 
\begin{equation*}
\begin{aligned}
    \bar{r}(x^{(1)}) &= (x^{(1)} - 1 - 0.1)\exp(0.2 x^{(1)})(0.2)(0.2) - 1 + 0.95 x^{(1)} \\
    &= 0.04(x^{(1)} - 1.1)\exp(0.2 x^{(1)}) - 1 + 0.95 x^{(1)}. 
\end{aligned}
\end{equation*}

In this example, we implement the deep actor-critic method~\cite{zhou2024solving} as the primal approach and compute the dual lower bound based on the generalized Hopf formula; the results are reported in Table~\ref{tab:bounds-aiyagari}. The lower bounds are computed using $M=500$ samples. The upper bounds are obtained by the deep actor-critic method and evaluated using $M_{\text{upper}} = 10000$ samples and $N_{T,\text{upper}} = 400$ time steps. 

\begin{table}[htbp!]
    \centering
    \begin{tabular}{ccccc}
        \hline 
        $(Z_0, A_0)$ & $V_{\text{lower}}$ (Hopf) & $V(0, Z_0, A_0)$ & $V_{\text{upper}}$ (Actor-Critic) & $95\%$ confidence interval \\
        \hline 
        $(1.0, 0.5)$ & $-0.2568 \pm 0.0014$ & $-0.2557$ & $-0.2555 \pm 0.0004$ & $[-0.2582, -0.2551]$ \\
        $(1.0, 1.0)$ & $-0.7565 \pm 0.0007$ & $-0.7557$ & $-0.7556 \pm 0.0007$ & $[-0.7572, -0.7549]$ \\
        $(1.0, 1.5)$ & $-1.2579 \pm 0.0015$ & $-1.2557$ & $-1.2555 \pm 0.0010$ & $[-1.2594, -1.2545]$  \\
        $(0.25, 1.0)$ & $-0.7909 \pm 0.0021$ & $-0.7897$ & $-0.7895 \pm 0.0007$ & $[-0.7930, -0.7888]$ \\
        $(0.75, 0.5)$ & $-0.7681 \pm 0.0009$ & $-0.7676$ & $-0.7673 \pm 0.0007$ & $[-0.7690, -0.7666]$ \\
        $(1.25, 0.5)$ & $-0.7442 \pm 0.0008$ & $-0.7432$ & $-0.7428 \pm 0.0007$ & $[-0.7450, -0.7421]$ \\
        \hline 
    \end{tabular}
    \caption{Numerical comparison of lower and upper bounds for $V(0,Z_0,A_0)$ with different $(Z_0,A_0)$.}
    \label{tab:bounds-aiyagari}
\end{table}

In Table~\ref{tab:bounds-aiyagari}, the reference value $V(0, Z_0, A_0)$ falls within $95\%$ confidence intervals, which validates the numerical reliability of the dual approach in continuous-time stochastic control problems. The gap between the primal upper bounds ($V_{\text{upper}}$) generated by the deep actor-critic method and the dual lower bounds ($V_{\text{lower}}$) computed via the generalized Hopf formula is narrow across all tested initial states. This small duality gap indicates the near-optimality of the learned control policies.

\section{Conclusion} \label{sec:conclusion}
Deep learning-based methods have emerged as powerful tools for solving stochastic optimal control problems, particularly in high-dimensional settings. However, quantifying the approximation errors inherent in these neural network architectures remains a significant challenge. This work addresses this issue by leveraging primal and dual formulations to compute reliable upper and lower bounds for the optimal value. The tightness of these bounds provides a practical indicator of solution accuracy, offering a robust alternative to the current lack of global theoretical convergence guarantees for deep learning solvers. 

Beyond the numerical construction of these dual bounds, we provide a rigorous proof of the generalized Hopf formula, resolving the conjecture in~\cite{chow2019algorithm} under mild conditions, namely the existence of an optimal control and the standard assumptions required for Pontryagin's maximum principle. This result establishes a formal theoretical foundation for developing curse-of-dimensionality-free algorithms to solve deterministic optimal control problems and state-dependent Hamilton--Jacobi equations. 

A key advantage of the proposed dual approaches is their flexibility: they are naturally compatible with any primal method capable of approximating the gradient-diffusion term $\nabla_xV(t,x)^\top \sigma(x)$. In this study, we validated this synergy using the deep BSDE~\cite{E2017deepbsde} and deep actor-critic~\cite{zhou2024solving} methods. A natural next step is to evaluate the robustness and efficiency of these dual approaches when integrated with a broader class of numerical solvers for stochastic control. 

Furthermore, extending this primal--dual framework to mean field games or mean field control problems represents a promising research direction. In such settings, deriving tight, computable error bounds could significantly enhance the reliability of deep learning algorithms applied to complex, large-population interactions governed by McKean--Vlasov dynamics.

\section*{Acknowledgements}
The authors would like to thank Prof. Xiaolu Tan for helpful discussions at an early stage of this project, in particular for suggestions related to Theorem~\ref{thm:spde}. 

\bibliographystyle{siam}
\bibliography{references}
\end{document}